\documentclass[reqno]{amsart}

\usepackage{graphicx}
\usepackage{amssymb}
\usepackage[all]{xy}
\usepackage{enumitem}

 \newtheorem{theorem}{Theorem}[section]
 \newtheorem{proposition}[theorem]{Proposition}
 
 \newtheorem{remark}[theorem]{Remark}
 
 \newtheorem{definition}[theorem]{Definition}
 
 \newtheorem{example}[theorem]{Example}

 \newcommand{\ZZ}{{\mathbb{Z}}}
 \newcommand{\QQ}{{\mathbb{Q}}}
 \newcommand{\RR}{{\mathbb{R}}}

  \newcommand{\gkn}{G_k(\RR^n)}
  \newcommand{\ow}{\overline{w}}
	\newcommand{\FF}{{\mathbb{F}}}
	\newcommand{\qbinom}[3]{\left[\begin{matrix}#1\\#2\end{matrix}\right]_{#3}}
	\newcommand{\sqbinom}[3]{\left[\begin{smallmatrix}#1\\#2\end{smallmatrix}\right]_{#3}}

\begin{document}

\title[Triangulations with few vertices]{How many simplices are needed\\ to triangulate a Grassmannian?} 

\author[Dejan Govc]{Dejan Govc}
\address{$^*$ Institute of Mathematics, University of Aberdeen, Aberdeen, UK}
\email{dejan.govc@gmail.com}

\author[Wac{\l}aw Marzantowicz]{Wac{\l}aw Marzantowicz$^{**}$}
\address{$^{**}$ \;Faculty of Mathematics and Computer Science, Adam Mickiewicz University 
of Pozna{\'n}, ul. Umultowska 87, 61-614 Pozna{\'n}, Poland.}
\email{marzan@amu.edu.pl}
\thanks{$^{**}$ Supported by the Research Grant NCN Grant  2015/19/B/ST1/01458
and Sheng 1 UMO-2018/30/Q/ST1/00228.}

\author[Petar Pave\v{s}i\'{c}]{Petar Pave\v si\'{c}$^{***}$}
\address{$^{***}$ Faculty of Mathematics and Physics, University of Ljubljana,
Jadranska 21,  1000 Ljubljana, Slovenija}
\email{petar.pavesic@fmf.uni-lj.si}
\thanks{$^{***}$ Supported by the Slovenian Research Agency program P1-0292 and grants N1-0083, 
N1-0064.}

\date{\today}

\begin{abstract}
We compute a lower bound for the number of simplices that are needed to triangulate the
Grassmann manifold  $\gkn$. In particular, we show that the number of top-dimensional
simplices grows exponentially with $n$. More precise estimates are given for $k=2,3,4$.
Our method can be used to estimate the minimal size of triangulations
for other spaces, like Lie groups, flag manifolds, Stiefel manifolds etc. 
\ \\[3mm]
{\it Keywords}: minimal triangulation, Grassmann manifold, cup-length, Lower Bound Theorem,
Manifold g-Theorem \\
{\it AMS classification: 57Q15, 52B05, 52B70, 14M15} 
\end{abstract}

\maketitle

%========================================================================================================================
%  INTRODUCTION
%========================================================================================================================

\section{Introduction}

Manifolds are often studied under the assumption that they can be triangulated, even when there are no explicit 
triangulations available. Indeed, triangulability allows the use of a whole spectrum of PL-techniques and
constructions. However, 
with an explicit triangulation at hand one can go even further and use computers to answer other specific
questions about the topology of the manifold. For example, there is a long-standing problem concerning 
the computation of rational Pontrjagin classes of combinatorial manifolds, see Gelfand-MacPherson
\cite{Gelfand-MacPherson} and Gaifullin \cite{Gaifullin}. In particular Gaifullin \cite{Gaifullin} described 
an explicit algorithm for the computation of the first Pontrjagin class of a manifold. The computation 
uses only the combinatorial structure of the triangulation and does not need any additional data. In spite
of that the explicit computations remain very hard. We are aware of only two examples of explicit computations: Milin \cite{Milin} used an earlier formula by Gabrielov-Gelfand-Losik
to compute the first Pontrjagin class of the complex projective plane, while Gorodkov
\cite{Gorodkov} used Gaifulin's formula to prove that a triangulation proposed by Brehm
and K\"uhnel \cite{Brehm-Kuhnel} actually represents a quaternionic 
projective plane. In both cases the main difficulty was represented by the size of the
triangulation. For
example in order to complete the computation Gorodkov \cite{Gorodkov} had to examine the links
of 3003 4-dimensional simplices.

In this paper we will consider the universal setting for characteristic classes,
Grassmann manifolds. 
We will show that the number of simplices in any triangulation of a Grassmann manifold must 
be huge, so that the computation of universal Pontrjagin classes using the combinatorial
formulas is probably beyond reach at this moment. 
For instance, we will show that every triangulation of the relatively small manifold
$G_3(\RR^9)$ requires thousands of facets (top-dimensional
faces) and hundreds of millions of simplices! 
This may look surprising in view of the fact that Grassmannian manifolds admit very efficient decompositions
into so-called Schubert cells. Specifically, the standard decomposition of $\gkn$ has ${n \choose k}$ cells (of
which only one 0-dimensional and one top-dimensional cell). Unfortunately, Schubert decomposition is a 
CW-decomposition 
that is very far from being regular. To refine it to a triangulation, one would need to find simplicial 
approximations of the attaching maps.

The problem of efficient or even minimal triangulations of manifolds has a long history but is still very 
much open. Minimal triangulations are known for spheres, surfaces, certain sphere bundles over circles ('Csazsar
 tori', see \cite{Kuhnel}) and for 
only a handful of other, low-dimensional examples (see \cite{Lutz}). 
There are also some coarse estimates: for example it is known that 
every triangulation of $\RR P^n$ requires at least $(n+1)(n+2)/2$ vertices, and on the other hand, there are explicit
triangulations of $\RR P^n$ with $2^{n+1}-1$ vertices (see \cite{Lutz}).

We have recently developed a method \cite{GMP} that uses the information about the cohomology 
ring of a space to derive lower bounds
for the number of vertices that are needed to triangulate that space. In this paper we apply 
the method
to estimate the size of triangulation of Grassmann manifolds. Our computation has three main 
ingredients.
\begin{enumerate}
\item R. Stong's \cite{Stong} determination of the height of the first Stiefel-Whitney 
class $w_1$ in $H^*(\gkn;\ZZ_2)$, and of
non-trivial products in the top dimension of $H^*(\gkn;\ZZ_2)$ for $k=2,3,4$.
\item Lower bounds for the number of vertices in a triangulation of a space whose cohomology 
admits certain 
non-trivial products \cite{GMP}.
\item The Lower  Bound Theorem (LBT) of Gromov \cite{Gromov} and Kalai \cite{Kalai} that
estimates the number of faces in a triangulation of a (pseudo)manifold with a given number 
of vertices. Moreover, we can obtain better estimates using results of Novik-Swartz
\cite{Novik-Swartz-I,Novik-Swartz-II,Novik-Swartz-III} and Adiprasito \cite{Adiprasito}.
\end{enumerate}

The paper is organized as follows. In the next section we give a brief overview of the 
cohomology of Grassmann manifolds including Stong's results and explicit description of 
the Poincar\'e polynomial. In Section 3  we use the non-trivial cup-products in 
$H^*(\gkn;\ZZ_2)$ 
to estimate the minimal number of vertices in a triangulation of $\gkn$. We first compute
lower bounds for arbitrary $k$, and then we obtain better estimates for $k=2,3,4$. 
Using LBT, this allows us to estimate the number of facets and of all simplices in a 
triangulation of $\gkn$. In the last Section 4 we obtain even better estimates in the 
orientable case where we can use the recently proved generalization of LBT. 

%========================================================================================================================
\section{Review of the cohomology of finite Grassmannians}
%========================================================================================================================

In this paper we denote by $\gkn$ the space of all $k$-dimensional linear subspaces of $\RR^n$. 
This space is a manifold (actually a non-singular algebraic variety) of dimension $k\cdot(n-k)$. The
correspondence between a subspace and its orthogonal complement determines a homeomorphism between $\gkn$ 
and $G_{n-k}(\RR^n)$, so we may assume without loss of generality  that $k\le n/2$.
The $\ZZ_2$-cohomology ring of $\gkn$ can be described as follows (cf. \cite{Stong}). 
Let $w_1,\ldots,w_k\in H^*(\gkn;\ZZ_2)$ denote the Stiefel-Whitney classes of the canonical
$k$-dimensional vector bundle over $\gkn$, and let $\ow_1,\ldots,\ow_{n-k}\in H^*(\gkn;\ZZ_2)$
denote the dual  Stiefel-Whitney classes 
(i.e., the Stiefel-Whitney classes of the orthogonal complement of the canonical bundle).
The Stiefel-Whitney classes and their duals are related by the formula
$$w\cdot\ow=(1+w_1+\ldots + w_k)\cdot(1+\ow_1+\ldots+\ow_{n-k})=1, $$ 
so the dual classes can be recursively expressed as polynomials in variables $w_1,\ldots,w_k$.
Then $H^*(\gkn;\ZZ_2)$ can be described as the quotient of the $\ZZ_2$-polynomial ring,
generated by the Stiefel-Whitney classes and their duals, modulo the relation $w\cdot \ow=1$:
$$ H^*(\gkn;\ZZ_2)\cong\ZZ_2[w_1,\ldots,w_k,\ow_1,\ldots,\ow_{n-k}]/(w\cdot \ow=1).$$

\begin{example}
In the cohomology of $G_2(\RR^3)$ we have the relation 
$$(1+w_1+w_2)\cdot (1+\ow_1+\ow_2+\ow_3)=1,$$ 
from which 
we obtain $\ow_1=w_1,\ \ow_2=w_1^2+w_2,\  \ow_3=w_1^3,\  w_1^4+w_1^2w_2+w_2^2=0$ and $w_1^3w_2=0$. Therefore
$$ H^*(G_2(\RR^5;\ZZ_2)\cong\ZZ_2[w_1,w_2]/(w_1^4+w_1^2w_2+w_2^2,\ w_1^3w_2).$$
\end{example}

In spite of the very explicit description, computations in $H^*(\gkn;\ZZ_2)$ are all but straightforward,
because it is difficult to determine whether a given polynomial in $w_1,\ldots,w_k$ is contained in the
ideal. R.~Stong \cite[p. 103-104]{Stong} proved the following results, which will be essential for our
computations.

\begin{proposition}
\label{prop:height w1}
If $k\ge 2$ and $2^s< n\le 2^{s+1}$ then the \emph{height} of $w_1$ (i.e., the maximal $m$, such that
$w_1^m\ne 0$ in $ H^*(\gkn;\ZZ_2)$) is 
$$\mathrm{height}(w_1)=\left\{\begin{array}{ll}
2^{s+1}-2 & k=2 \mathrm{\ or\ } k=3, n=2^s+1\\
2^{s+1}-1 & \mathrm{otherwise}.
\end{array}\right.$$
\end{proposition}

\begin{proposition}
\label{prop:products in g3n}
Let $2^s< n\le 2^{s+1}$. Then the following classes in $H^{3(n-3)}(G_3(\RR^n);\ZZ_2)$ are non trivial
($1\le p$, $0<t<2^{p-1}$):
$$\begin{array}{ll}
\mathrm{if\ } n=2^{s+1}-2^p+1, & \mathrm{then\ \ } w_1^{2^{s+1}-2} \cdot w_2^{2^{s+1}-3\cdot 2^{p-1}-2}\ne 0; \\
\mathrm{if\ } n=2^{s+1}-2^p+1+t, & \mathrm{then\ \ } w_1^{2^{s+1}-1} \cdot w_2^{2^{s+1}-3\cdot 2^{p-1}-1}\cdot w_3^{t-1}\ne 0; \\
\mathrm{if\ } n=2^{s+1}, & \mathrm{then\ \ } w_1^{2^{s+1}-1} \cdot w_2^{2^{s+1}-4}\ne 0.
\end{array}$$
\end{proposition}

\begin{proposition}
\label{prop:products in g4n}
Let $2^s< n\le 2^{s+1}$. Then the following classes in $H^{4(n-4)}(G_4(\RR^n);\ZZ_2)$ are non trivial
(\,$0\le r<s$, $0\le t<2^r$):
$$\begin{array}{ll}
\mathrm{if\ } n=2^s+1, & \mathrm{then\ \ } w_1^{2^{s+1}-2} \cdot w_2^{2^s-5}\ne 0, \ 
w_1^{2^{s+1}-1} \cdot w_2^{2^s-7}\cdot w_3\ne 0;\\
\mathrm{if\ } n=2^s+2^r+1+t, & \mathrm{then\ \ } w_1^{2^{s+1}-2} \cdot w_2^{2^s+2^{r+1}-5}\cdot w_4^t\ne 0,\\

 & \mathrm{also\ } w_1^{2^{s+1}-1} \cdot w_2^{2^s+2^{r+1}-7}\cdot w_3\cdot w_4^t\ne 0 \mathrm{\ if\ } r>0 ;
\end{array}$$
\end{proposition}

\begin{example}
For $G_3(\RR^9)$ we may write $9=2^4-2^3+1$. Then the height of $w_1$ is $2^{3+1}-2=14$ by 
Proposition \ref{prop:height w1} and $w_1^{14}\cdot w_2^2\ne 0$ in $H^{18}(G_3(\RR^9);\ZZ_2)$ 
by Proposition \ref{prop:products in g3n}.
\end{example}

We will also need the explicit computation of the Poincar\'e polynomial for the rational cohomology of
Grassmannians. Following Casian-Kodama \cite[Theorem 5.1]{Casian-Kodama}: 
the Poincar\'e polynomial for $G_k(\RR^n)$ in the case $(k,n)=(2j,2m),(2j,2m+1)$ or $(2j+1,2m+1)$ is given by
\[
P_{k,n}(t)=\qbinom{m}{j}{t^4}
\]
and for $(k,n)=(2j+1,2m)$ it is given by
\[
P_{k,n}(t)=(1+t^{2m-1})\qbinom{m-1}{j}{t^4}.
\]
In both cases, $\sqbinom{n}{k}{q}$ denotes the $q$-analog of the binomial coefficient. More 
explicitly, it is expressed by the formula
\[
\qbinom{n}{k}{q}=\frac{\prod_{i=1}^n(1-q^i)}{\prod_{i_1=1}^{k}(1-q^{i_1})\prod_{i_2=1}^{n-k}(1-q^{i_2})}.
\]
These expressions have various nice properties. For example, they satisfy a recurrence relation similar to the one for binomial coefficients:
\[
\qbinom{n+1}{k}{q}=q^k\qbinom{n}{k}{q}+\qbinom{n}{k-1}{q}.
\]
For our purposes, the main property of the symbol $\sqbinom{n}{k}{q}$ (proved easily 
by induction) is that it is a polynomial of degree $k(n-k)$ in $q$ 
and that all of its coefficients are strictly positive integers.

%========================================================================================================================
\section{Lower bounds for the number of simplices}
%========================================================================================================================

We will first use the results of \cite{GMP} to estimate the minimal number of vertices in a triangulation
of $\gkn$. Following \cite{GMP},  let us denote by $\Delta(X)$ the minimal number of vertices in a triangulation
of a triangulable space $X$. It will be convenient to state the results in the following form.

\begin{proposition} (\cite[Theorem 2.2]{GMP})
\label{prop:height estimate}
If $x^n\ne 0$ for some $x\in H^1(X)$ (with arbitrary coefficients), then 
$$\Delta(X)\ge \frac{(n+1)(n+2)}{2}.$$
\end{proposition}

\begin{proposition} (\cite[Theorem 3.5]{GMP})
\label{prop:ct estimate}
If $x_1\cdot x_2\cdot\ldots\cdot x_n\ne 0$ for some (homogeneous and not all of same dimension) elements
of $\widetilde H^*(X)$, then 
$$\Delta(X)\ge |x_1|+2|x_2|+\ldots +n|x_n|+(n+2)$$
(where $|x_i|$ denotes the dimension of $x_i$).
\end{proposition}

Let us begin with the generic case.

\begin{theorem}
\label{thm:height estimate}
Assume $k\ge 4$, and let $2^s< n\le 2^{s+1}$ for some $s$. Then
$$\Delta(G_k(\RR^n))\ge 2^s(2^{s+1}+1).$$
As a consequence, $\Delta(G_k(\RR^n))$ increases (at least) as a quadratic function of $n$.
\end{theorem}
\begin{proof}
By Proposition  \ref{prop:height w1} we have $w_1^{2^{s+1}-1}\ne 0$, and then by
Proposition \ref{prop:height estimate} 
$$\Delta(G_k(\RR^n))\ge \frac{2^{s+1}(2^{s+1}+1)}{2}=2^s(2^{s+1}+1).$$
For the second statement, observe that $n\le 2^{s+1}\le 2n-2$, therefore 
$$\frac{n(n+1)}{2}\le 2^s(2^{s+1}+1)\le (n-1)(2n-1).$$
\end{proof}

For example, by the above theorem $\Delta(G_3(\RR^8))\ge 36$, so the quadratic growth does not automatically
imply that the resulting simplicial complex must be very big. However, by applying the full strength of 
Stong's results we will be able to obtain estimates that 
are still quadratic in $n$ but that give much bigger lower bounds for the number of vertices.

The simplest case is given by 2-dimensional subspaces of $\RR^n$. It follows from  the description of 
$H^*(G_2(\RR^n);\ZZ_2)$ and from Proposition \ref{prop:height w1} that for $2^s<n\le 2^{s+1}$ 
the height of $w_1$ is $2^{s+1}-2$ and that the top-dimensional cohomology class in 
$H^{2(n-2)}(G_2(\RR^n);\ZZ_2)$ is 
$$w_1^{2^{s+1}-2}\cdot w_2^{n-1-2^s}\ne 0.$$
By performing the summation from Proposition \ref{prop:ct estimate} we obtain the following result.

\begin{theorem}
\label{thm:vertices gkn}
Let $2^s< n\le 2^{s+1}$. Then 
$$\Delta(G_2(\RR^n))\ge (n-2)^2+2^s(2n-2^s-1).$$
In particular, if $n=2^s$, then $\Delta(G_2(\RR^n))\ge  \frac{7}{4}n^2-\frac{9}{2}n+4$, and if 
$n=2^s+1$, then $\Delta(G_2(\RR^n))\ge  2n^2-5n+4$.
\end{theorem}

Here we clearly see that as a function of $n$ the estimates that we derive from Stong's results are 
comparatively better when $n$ is slightly bigger than a power of 2. The obvious cause is that the 
height of $w_1$ almost doubles when $n$ passes the power of two (cf. Proposition \ref{prop:height w1}). 
The computation for Grassmann 
manifolds of 3- and 4- dimensional subspaces of $\RR^n$ follows the same lines: one uses Stong's 
results to determine the longest non-trivial products in cohomology and then applies Proposition
\ref{prop:ct estimate} to compute the lower bound for the number of vertices in a triangulation. 
However, the study of all sub-cases would give complicated and irregular results.  To avoid this 
difficulty, we will compute only the limiting cases, when  $n$ is a power of two,
and when $n$ is by one bigger than a power of two. 

\begin{theorem}\ \\[-6mm]
\label{thm:vertices g3n}
\begin{enumerate}
\item If $n=2^s$, then $\Delta(G_3(\RR^n))\ge 3n(n-5)+\frac{n^2-n}{2}+17$.
\item If $n=2^s+1$, then $\Delta(G_3(\RR^n))\ge 4n(n-5)+\frac{(n-1)^2}{4}+25$.
\end{enumerate}
\end{theorem}
\begin{proof}
If $n=2^s$, then Proposition \ref{prop:products in g3n} implies that the longest non-trivial product in
$H^*(G_3(\RR^n);\ZZ_2)$ is $w_1^{n-1}\cdot w_2^{n-4}\ne 0$. Then we have by Proposition \ref{prop:ct estimate}
$$\Delta(G_3(\RR^n))\ge ((1+\ldots+(n-1))+2(n+\ldots+(2n-5))+(2n+3)=\frac{7}{2}n^2-\frac{31}{2}n+17.$$
The last expression can be conveniently restated as in (1) to allow comparison with the estimate (2).

If $n=2^s+1$, then Proposition \ref{prop:products in g3n} gives $w_1^{2n-4}\cdot w_2^{\frac{n-5}{2}}\ne 0$
as the longest non-trivial product. The summation as in (1) yields the estimate 
$$\Delta(G_3(\RR^n))\ge \frac{17}{4}n^2-\frac{41}{2}n+\frac{101}{4}=4n(n-5)+\frac{(n-1)^2}{4}+25.$$
\end{proof}

\begin{example}
\label{ex:g3n}
By the above Theorem we have 
$$\Delta(G_3(\RR^8))\ge 117 \ \ \mathrm{and}\ \ \Delta(G_3(\RR^9))\ge 185.$$
At the next power of two the jump in the estimate is much bigger:
$$\Delta(G_3(\RR^{16}))\ge 665 \ \ \mathrm{and}\ \ \Delta(G_3(\RR^{17}))\ge 905.$$
\end{example}

We conclude this part with the estimates for the limit cases of $G_4(\RR^n)$.

\begin{theorem}\ \\[-6mm]
\label{thm:vertices g4n}
\begin{enumerate}
\item If $n=2^s$, then $\Delta(G_4(\RR^n))\ge 6n(n-6)-\frac{3}{8}(n^2+2n)+57$.
\item If $n=2^s+1$, then $\Delta(G_4(\RR^n))\ge 7n(n-7)+3n+89$.
\end{enumerate}
\end{theorem}
\begin{proof}
By Proposition \ref{prop:products in g4n} the cohomology products that give the highest estimates 
are $w_1^{n-1}\cdot w_2^{n-7}\cdot w_3\cdot w_4^{\frac{n}{4}-1}\ne 0$ for $n=2^s$, and 
$w_1^{2n-3}\cdot w_2^{n-8}\cdot w_3\ne 0$ for $n=2^s+1$. The application of Proposition 
\ref{prop:ct estimate} then gives the stated results.
\end{proof}

\begin{example}
\label{ex:g4n}
We have the following estimates:
$$\Delta(G_4(\RR^8))\ge 123, \ \ \ \ \Delta(G_4(\RR^9))\ge 242$$
and 
$$\Delta(G_4(\RR^{16}))\ge 909, \ \ \ \ \Delta(G_4(\RR^{17}))\ge 1330.$$
Observe that the increase of the estimate from $n=16$ to $n=17$ covers almost two-thirds of the increase 
from $n=9$ to $n=16$. However, bear in mind that these values are artefacts of the method and do not
reflect necessarily the actual increase of $\Delta(G_4(\RR^n))$.
\end{example}

We are now ready to estimate the number of facets and of all simplices that are needed to 
triangulate a Grassmann manifold. To this end we apply the Lower Bound Theorem, which we state 
following Kalai \cite{Kalai}. Note that LBT does not take into account the homology of 
the manifold. We will be able to obtain better estimates by using a generalized version of 
LBT which is however proved only for orientable manifolds (in our case for $\gkn$ with
$n$ even). 

\begin{theorem}
\label{thm:LBT}
Let $K$ be a triangulation of a $d$-dimensional closed manifold, and denote by $f_i$, $i=0,\ldots,d$ 
the number of $i$-dimensional simplices in $K$. Then
$$f_i\ge f_0\cdot{d+1 \choose i}-i\cdot{d+2\choose i+1}\ \ \textrm{for\ }i=0,\ldots,d-1$$
and
$$f_d\ge f_0\cdot d- (d+2)(d-1).$$
Moreover, by adding up all inequalities we obtain an estimate for the total number of simplices in $K$:
$$(f_0+\ldots+f_d)\ge 2[(f_0-d)(2^{d+1}-1)+1].$$
\end{theorem}

Observe that the number of all simplices always increases exponentially with the dimension.
In fact, even the minimal triangulation of the simplest closed manifold, the $d$-dimensional
sphere, has $d+2$ vertices and $2^{d+2}-2$ simplices. 
However, we will show that the number of simplices that are needed to triangulate a
Grassmannian of comparable dimension is several orders of magnitude bigger. 

We will state as a Theorem only the estimate for the generic case that follows from the computation 
of $\textrm{height}(w_1)$, and is valid for all Grassmann manifolds. Better estimates for small $k$ will 
be relegated to  examples. It turns out to be more convenient to write $G_k(\RR^{n+k})$ because 
its dimension is expressed as $kn$. 

\begin{theorem}
\label{thm:facets}
Every triangulation of the Grassmann manifold $G_k(\RR^{n+k})$ must have at least
$$\frac{k}{2}\cdot(n^3+n^2)+\frac{k(k+2)(k-1)}{2}\cdot n+2$$
facets. Furthermore, every triangulation of $G_k(\RR^{n+k})$ must have at least
$$[(n+k)(n+k+1)-2kn]\cdot(2^{kn+1}-1)$$
simplices.
\end{theorem}
\begin{proof}
We have proved in Theorem \ref{thm:height estimate} that every triangulation of $G_k(\RR^{n+k})$
has at least $\frac{(n+k)(n+k+1)}{2}$ vertices. Note that for $k<4$ we obtained even better estimates
so the above formula is actually valid for all $k$. The statement follows by Theorem \ref{thm:LBT}.
\end{proof}

To corroborate our initial claim about the size of triangulations of Grassmann manifolds we will
compute the lower bounds for some small Grassmannians.

\begin{example}
$G_3(\RR^9)$ is 18-dimensional and by Example \ref{ex:g3n} every triangulation requires at least 
185 vertices. As a consequence, every triangulation of $G_3(\RR^9)$ must have at least
$$185\cdot 18-(18+2)\cdot(18-1)=2990$$
facets and at least 
$$2((185-18)\cdot(2^{19}-1)+1)> 175\cdot 10^6$$
simplices.

$G_4(\RR^9)$ is 20-dimensional and $\Delta(G_4(\RR^9))\ge 242$ by Example \ref{ex:g4n}. 
Therefore, every triangulation of $G_4(\RR^9)$ requires more than 4422 facets and more than 
$930\cdot 10^6$ simplices. The number of 4-dimensional simplices, whose links should be examined
to compute the first rational Pontrjagin class by means of Gaifullin's formula exceeds 1.3 million. 
\end{example}

Rounding a bit, we can say that one needs more than a billion simplices to build
a manifold that carries only the first four Stiefel-Whitney classes (and two relations among them).

%========================================================================================================================
\section{Improving the bounds using the Manifold $g$-Theorem}
%========================================================================================================================

Note that LBT as stated above does not take into account the homology of the manifold, so we are led to consider stronger results. Much of the research in enumerative combinatorics of simplicial complexes has been guided by various versions of the so called $g$-Conjecture. These are a far-reaching generalization of the LBT. Of particular interest to us is the so called Manifold $g$-Conjecture (see \cite[Section 4]{Klee-Novik}), as it comes with a version of LBT for manifolds that takes into account the Betti numbers. Recently, Adiprasito has published a preprint \cite{Adiprasito}, whose results combined with the work of Novik and Swartz \cite{Novik,Novik-Swartz-II,Novik-Swartz-III,Swartz} imply that the Manifold $g$-Conjecture is indeed true.

Before stating the Manifold $g$-Theorem, we need some background. The following definition is essentially due to McMullen and Walkup \cite{McMullen-Walkup}.

\begin{definition}
Let $K$ be a $d$-dimensional simplicial complex and $(f_0,f_1,\ldots,f_d)$ its face vector, with $f_i$ for $i=0,\ldots,d$ as defined in the previous section. Additionally, define $f_{-1}=1$ and $f_i=0$ for $i>d$ and $i<-1$. Then, the \emph{$h$-vector} associated to $K$ is the vector $(h_0,h_1,\ldots,h_{d+1})$, where we define:
\[
h_j=\sum_{i=0}^j(-1)^{j-i}\binom{d+1-i}{d+1-j}f_{i-1},\qquad j\in\ZZ.
\]
Note that $h_j=0$ for $j<0$ and $j>d+1$. The $f$-vector can be recovered by inverting the equations\footnote{This can be done e.g. by using generating functions.}:
\[
f_{i-1}=\sum_{j=0}^i\binom{d+1-j}{d+1-i}h_j,\qquad i\in\ZZ.
\]
Similarly, the \emph{$g$-vector} associated to $K$ is the vector $(g_0,g_1,\ldots,g_{\lfloor\frac{d+1}2\rfloor})$ where we define:
\[
g_j=\sum_{i=0}^j(-1)^{j-i}\binom{d+2-i}{d+2-j}f_{i-1},\qquad j\in\ZZ.
\]
Note that $g_j=0$ for $j<0$ and $j>d+2$. Also, one may verify by explicit calculation that $g_j=h_j-h_{j-1}$ for all $j\in\ZZ$ and recover the $f$-vector by inverting:
\[
f_{i-1}=\sum_{j=0}^i\binom{d+2-j}{d+2-i}g_j,\qquad i\in\ZZ.
\]
\end{definition}

The original reason for the introduction of $h$-vectors in \cite{McMullen-Walkup} was to formulate the Generalized Lower Bound Conjecture (now Theorem \cite{Stanley,Murai-Nevo} and thus abbreviated as GLBT) for simplicial polytopes, whose first part states that $h_0\leq\ldots\leq h_{\lfloor\frac{d+1}2\rfloor}$ and second part concerns the cases of equality. Furthermore, it is shown in \cite{McMullen-Walkup} that the first part of GLBT together with the Dehn-Sommerville relations $h_j=h_{d+1-j}$, $j\in\ZZ$, implies the LBT for simplicial polytopes (the latter was proved in \cite{Barnette1,Barnette2}).

There is a huge body of work generalizing these ideas beyond the case of simplicial polytopes to the case of simplicial spheres and simplicial manifolds (see \cite{Klee-Novik} for a good survey of the field). The Manifold $g$-Theorem, which is of interest to us, is concerned with so called $h''$-vectors, which are a version of $h$-vectors more attuned to the case of manifolds, taking their reduced Betti numbers $\beta_i$, $i=0,\ldots,d$, into account. These are defined as follows \cite{Novik-Swartz-I}:
\[
h''_j=h_j-\binom{d+1}{j}\sum_{i=0}^j(-1)^{j-i}\beta_{i-1},\qquad 0\leq j\leq d,
\]
and\footnote{Note that this second equation, as stated in \cite{Klee-Novik}, seems to contain a typo.}
\[
h''_{d+1}=h_{d+1}-\sum_{i=0}^d(-1)^{d+1-i}\beta_{i-1}.
\]
As in the case of $h$-vectors, we extend to all integer indices by defining $h''_j=0$ for $j<0$ and $j>d+1$.

We can now state the theorem. In doing so, we closely follow \cite[Theorem 6]{Klee-Novik}; however all references to the WLP can be omitted from the statement of the Manifold $g$-Theorem, as thanks to \cite{Adiprasito} it is now known that all homology spheres have the WLP:

\begin{theorem}
\label{thm:h''-vectors}
Suppose $\FF$ is a field with infinitely many elements. Let $M$ be a connected $d$-dimensional orientable triangulated $\FF$-homology manifold without boundary. Suppose $\beta_i$, $i=0,\ldots,d$, are the reduced Betti numbers of $M$ with respect to $\FF$. Then the following conditions hold:
\begin{enumerate}
\item Dehn-Sommerville: $h''_j=h''_{d+1-j}$ for $j=0,\ldots,d+1$,
\item Nonnegativity: $h''_j\geq 0$ for $j=0,\ldots,d+1$,
\item The Manifold $g$-Theorem:
\begin{enumerate}
\item $h''_0\leq h''_1 \ldots\leq h''_{\lfloor\frac{d+1}2\rfloor}$,
\item $(g''_0,g''_1,g''_2,\ldots,g''_{\lfloor\frac{d+1}2\rfloor})$, where $g''_j:=h''_j-h''_{j-1}$, $j\in\ZZ$, is an M-sequence\footnote{See \cite[Section 6.1]{Klee-Novik} for the definition of M-sequence.}.
\end{enumerate}
\end{enumerate}
Finally, nonnegativity continues to hold without the orientability assumption.
\end{theorem}

Note that $h''_{j-1}\leq h''_j$ is equivalent to $g''_j\geq 0$. In fact, the following stronger version of the Manifold $g$-Theorem is true \cite[Theorem 13]{Klee-Novik}, where all references to the WLP can again be omitted thanks to \cite{Adiprasito}:

\begin{theorem}
\label{thm:SMGC}
Suppose $\FF$ and $M$ satisfy the assumptions of Theorem \ref{thm:h''-vectors}. Define $\tilde g_j=g''_j-\binom{d+1}{j-1}\beta_{j-1}$. Then:
\begin{enumerate}
\item $\tilde g_j$ is nonnegative for $j=0,1,\ldots,\lfloor\frac{d+1}2\rfloor$,
\item $(\tilde g_0,\tilde g_1,\tilde g_2,\ldots,\tilde g_{\lfloor\frac{d+1}2\rfloor})$ is an M-sequence.
\end{enumerate}
\end{theorem}

Note that the first part of the Manifold $g$-Theorem (nonnegativity of $g''$-vectors) can be seen as a version of the GLBT for manifolds. As such, it should also have a corresponding version of the LBT, i.e. a corresponding set of inequalities for the $f$-vectors. Such a statement is undoubtedly well-known to enumerative combinatorialists, but we have been unable to find it stated explicitly in the literature, so we include a proof for completeness. Note that the bound for each $f$-number only differs from the bound of Theorem \ref{thm:LBT} by an additional summand depending on the Betti numbers.

\begin{theorem}
\label{thm:LBTM}
Under the assumptions of Theorem \ref{thm:h''-vectors}, we have the following bounds:
$$f_i\ge f_0\cdot{d+1 \choose i}-i\cdot{d+2\choose i+1}+{d+1\choose i+1}\sum_{j=0}^i{i\choose j}\beta_j\ \ \textrm{for\ }i=0,\ldots,d-1$$
and
$$f_d\ge f_0\cdot d- (d+2)(d-1)+\sum_{j=0}^{d-1}{d\choose j}\beta_j.$$
Moreover, by adding up all inequalities we obtain an estimate for the total number of simplices in the triangulation of $M$:
$$(f_0+\ldots+f_d)\ge 2[(f_0-d)(2^{d+1}-1)+1]+\sum_{j=0}^{d-1}\left[\binom{d}{j}+\sum_{i=j}^{d-1}\binom{d+1}{i+1}\binom{i}{j}\right]\beta_j.$$
\end{theorem}

\begin{proof}
This follows by a simple modification of the original proof of McMullen and Walkup \cite{McMullen-Walkup} that GLBT implies LBT. For the convenience of the reader we provide the full argument here. First define the modified face numbers as the inversion of the $h''$-numbers:
\[
f''_{i-1}:=\sum_{j=0}^i\binom{d+1-j}{d+1-i}h''_j,\qquad i\in\ZZ.
\]
These can be now be expressed in terms of the usual $f$-numbers, by substituting the expressions for $h''_j$ in terms of $h_j$ and the Betti numbers into the equation. For $0\leq i\leq d$, we proceed as follows:
\begin{align*}
f''_{i-1}&=\sum_{j=0}^i\binom{d+1-j}{d+1-i}h_j-\sum_{j=0}^i\sum_{k=0}^j\binom{d+1-j}{d+1-i}\binom{d+1}{j}(-1)^{j-k}\beta_{k-1}\\
&=f_{i-1}-\sum_{k=0}^i\left[\sum_{j=k}^i\binom{d+1-j}{d+1-i}\binom{d+1}{j}(-1)^{j}\right](-1)^k\beta_{k-1}
\end{align*}
With some rearranging, the sum in the square brackets can be expressed as follows:
\begin{align*}
\sum_{j=k}^i\binom{d+1-j}{d+1-i}\binom{d+1}{j}(-1)^{j}&=\binom{d+1}{i}\sum_{j=k}^i\binom{i}{j}(-1)^{j}\\
&=\binom{d+1}{i}\cdot(-1)^k\binom{i-1}{k-1}.
\end{align*}
Therefore,
\[
f''_{i-1}=f_{i-1}-\binom{d+1}{i}\sum_{k=0}^i\binom{i-1}{k-1}\beta_{k-1}.
\]
For $i=d+1$, exactly the same computation works, except that $\beta_d$ is replaced by $0$. This gives the final result
\[
f''_d=f_d-\sum_{k=0}^{d}\binom{d}{k-1}\beta_{k-1}.
\]
Note that the algebraic relations between the $f''$- and $h''$-numbers are the same as those between the usual $f$- and $h$-numbers. In particular, we can now also express the $f''$-numbers in terms of the $g''$-numbers as follows:
\[
f''_{i-1}:=\sum_{j=0}^i\binom{d+2-j}{d+2-i}g''_j,\qquad i\in\ZZ.
\]
Note that the upper summation limit can be replaced by $d+2$ without changing the value, so we can write:
\[
f''_i=\sum_{j=0}^{d+2}\binom{d+2-j}{d+1-i}g''_j.
\]
Now we use the Dehn-Sommerville relations for the $h''$-numbers. In terms of the $g''$-numbers these state that $g''_j=-g''_{d+2-j}$, $j\in\ZZ$. Therefore we can express $g''_j$ for $j\geq\lfloor\frac{d+2}{2}\rfloor+1$ in terms of the $g''_j$ for $j\leq\lfloor\frac{d+2}{2}\rfloor$ to obtain:
\[
f''_i=\sum_{j=0}^{\lfloor\frac{d+2}2\rfloor}\left[\binom{d+2-j}{d+1-i}-\binom{j}{d+1-i}\right]g''_j.
\]
Note that for even $d$ this uses the fact that $g''_{\frac{d}2+1}=0$. Now observe that all the terms in the obtained sum are nonnegative, therefore we have:
\[
f''_i\geq\sum_{j=0}^1\left[\binom{d+2-j}{d+1-i}-\binom{j}{d+1-i}\right]g''_j.
\]
Substituting the expression for $f''_i$ in terms of $f_i$ and the Betti numbers into the inequality, and observing that $g''_0=f_{-1}=1$ and $g''_1=f_0-d-2$ concludes the proof.
\end{proof}

The Strong Manifold $g$-Theorem comes with an even stronger version of the GLBT:

\begin{theorem}
\label{thm:SLBTM}
Under the assumptions of Theorem \ref{thm:h''-vectors}, we have the following bounds:
\begin{multline*}
f_i\ge f_0\cdot{d+1 \choose i}-i\cdot{d+2\choose i+1}+{d+1\choose i+1}\sum_{j=0}^i{i\choose j}\beta_j\\+\sum_{j=2}^{\lfloor\frac{d+2}2\rfloor}\left[\binom{d+2-j}{d+1-i}-\binom{j}{d+1-i}\right]\binom{d+1}{j-1}\beta_{j-1}\ \ \textrm{for\ }i=0,\ldots,d-1
\end{multline*}
and
$$f_d\ge f_0\cdot d- (d+2)(d-1)+\sum_{j=0}^{d-1}{d\choose j}\beta_j+\sum_{j=2}^{\lfloor\frac{d+2}2\rfloor}(d+2-2j)\binom{d+1}{j-1}\beta_{j-1}.$$
\end{theorem}

\begin{proof}
As the starting point, we use the following expression that we derived in the proof of Theorem \ref{thm:LBTM}:
\[
f''_i=\sum_{j=0}^{\lfloor\frac{d+2}2\rfloor}\left[\binom{d+2-j}{d+1-i}-\binom{j}{d+1-i}\right]g''_j.
\]
Now, by Theorem \ref{thm:SMGC}, we know that $g''_j\geq\binom{d+1}{j-1}\beta_{j-1}$. Therefore:
\begin{multline*}
f''_i\geq\sum_{j=0}^1\left[\binom{d+2-j}{d+1-i}-\binom{j}{d+1-i}\right]g''_j\\+\sum_{j=2}^{\lfloor\frac{d+2}2\rfloor}\left[\binom{d+2-j}{d+1-i}-\binom{j}{d+1-i}\right]\binom{d+1}{j-1}\beta_{j-1}.
\end{multline*}
Comparing this with the bound of Theorem \ref{thm:LBTM} gives the result.
\end{proof}

\begin{remark}
To use these results, the field of coefficients should be infinite, so we cannot use $\ZZ_2$-cohomology. Using rational cohomology therefore seems like a natural candidate. Furthermore, because of the orientability hypothesis, we can only use the result for $G_k(\RR^n)$ where $n$ is even. However, the improvement in that case seems significant.
\end{remark}

\begin{example}
The rational cohomology of $G_3(\RR^8)$, which is $15$-dimensional, is given by \cite{Wendt}:
\[
H^*(G_3(\RR^8);\QQ)=\QQ[p_1,r]/(p_1^3,r^2),\qquad \deg p_1 = 4,\deg r = 7.
\]
The homogeneous generators in the various degrees are $1,p_1,r,p_1^2,p_1r,p_1^2r$ and the corresponding Poincar\'e polynomial is $1+t^4+t^7+t^8+t^{11}+t^{15}$. Example \ref{ex:g3n} implies that $f_0\leq 117$. This gives the following results (the first column gives the lower bounds using Theorem \ref{thm:LBT}, the second column gives the lower bounds using Theorem \ref{thm:LBTM} and the third one gives the lower bounds using Theorem \ref{thm:SLBTM}):

\begin{center}
\begin{tabular}{|l||r|r|r|}
\hline
&$g$&$g''$&$\tilde g$\\
\hline
\hline
$f_0$&$117$&$117$&$117$\\
\hline
$f_1$&$1736$&$1736$&$1736$\\
\hline
$f_2$&$12680$&$12680$&$12680$\\
\hline
$f_3$&$58380$&$58380$&$58380$\\
\hline
$f_4$&$188188$&$192556$&$194376$\\
\hline
$f_5$&$449176$&$489216$&$511056$\\
\hline
$f_6$&$820248$&$991848$&$1111968$\\
\hline
$f_7$&$1168310$&$1631630$&$2043470$\\
\hline
$f_8$&$1311310$&$2215070$&$3207490$\\
\hline
$f_9$&$1163448$&$2532816$&$4294576$\\
\hline
$f_{10}$&$813176$&$2451176$&$4773496$\\
\hline
$f_{11}$&$442988$&$1946308$&$4186728$\\
\hline
$f_{12}$&$184380$&$1189020$&$2721460$\\
\hline
$f_{13}$&$56680$&$512200$&$1214720$\\
\hline
$f_{14}$&$12136$&$136936$&$330376$\\
\hline
$f_{15}$&$1517$&$17117$&$41297$\\
\hline
\hline
$\sum$&$6684470$&$14378806$&$24703926$\\
\hline
\end{tabular}
\end{center}
\end{example}

Note that this approach can be used anytime we know the Poincar\'e polynomial for the cohomology of the space we are interested in. However, to use the (Strong) Manifold $g$-Theorem, it seems to be essential that the manifold in question is orientable. We will now show that the number of facets in a triangulation of $G_2(\RR^n)$ for even $n$ (so that the corresponding Grassmann manifold is orientable) must grow not just as a cubical function of $n$ as seen in Theorem \ref{thm:facets}, but in fact as an exponential function of $n$.

\begin{theorem}
As above, let $f_d$ be the number of facets in a minimal triangulation of $G_2(\RR^n)$. Then assuming $n$ is even, we have 
\[
f_d\geq 4^{n-3}+(-1)^{\frac{n}{2}+1}2^{n-3}+(n-2)(n^2-3n+6).
\]
In particular, it grows (at least) as an exponential function in $n$.
\end{theorem}

\begin{proof}
Let $n=2m$. The Poincar\'e polynomial of $G_2(\RR^n)$ is
\[
P(x)=1+x^4+\ldots+x^{4m-4}.
\]
The bound for the generic case we stated in Theorem \ref{thm:facets} reduces to
\[
f_d\geq(n-2)(n^2-3n+6)+2
\]
in the case at hand. Theorem \ref{thm:LBTM} allows us to improve this by an additional summand coming from the Betti numbers. This summand is equal to:
\[
\sum_{j=0}^{d-1}{d\choose j}\beta_j=\sum_{k=1}^{m-2}{4m-4\choose 4k}=4^{2m-3}+(-1)^{m+1}2^{2m-3}-2.
\]
Summing everything together results in the stated bound.
\end{proof}

% COMMENTING THIS PART OUT, SINCE A STRONGER RESUlT IS POSSIBLE USING JUST $h''_j \geq 0$.
%
%\begin{theorem}
%Suppose $3\leq k\leq\lfloor\frac n2\rfloor$ and $d=k(n-k)$. Let $f_d$ be the number of facets in a minimal triangulation of $G_k(\RR^n)$. Then, for even $n$, $f_d$ grows (at least) as an exponential function in $n$.
%\end{theorem}
%
%\begin{proof}
%Recall that the $q$-analogs of binomial coefficients have positive integer coefficients. Using the explicit expression of $P_{(k,n)}(t)$, this means in particular that the $(2n-8)$-th rational Betti number of $G_k(\RR^n)$ is non-zero. Therefore, using the bound from Theorem \ref{thm:LBTM} we have:
%\[
%f_d\ge f_0\cdot d- (d+2)(d-1)+\sum_{j=0}^{d-1}{d\choose j}\beta_j\ge\binom{k(n-k)}{2n-8},
%\]
%which is exponential in $n$.
%\end{proof}

Finally, we show that much less than the full strength of the Manifold $g$-Theorem is sufficient to establish exponential growth of the number of facets. In fact, the nonnegativity of the $h''$-numbers is completely sufficient to establish it. In particular, no assumption of orientability is needed.

\begin{theorem}
Suppose $3\leq k\leq\lfloor\frac n2\rfloor$ and $d=k(n-k)$. Let $f_d$ be the number of facets in a minimal triangulation of $G_k(\RR^n)$. Then, $f_d$ grows (at least) as an exponential function in $n$.
\end{theorem}

\begin{proof}
We begin by expressing $f_d$ in terms of $h$-numbers using the inverse formula:
\[
f_d=\sum_{j=0}^{d+1}h_j.
\]
Nonnegativity of the $h''$-numbers can be restated as:
\[
h_j\geq\binom{d+1}{j}\sum_{i=0}^j(-1)^{j-i}\beta_{i-1}.
\]
Note that this is valid for $j=d+1$ as well. Plugging this into the first equation, we have:
\begin{align*}
f_d&\geq\sum_{j=0}^{d+1}\binom{d+1}{j}\sum_{i=0}^j(-1)^{j-i}\beta_{i-1}\\
&=\sum_{i=0}^{d+1}\left[\sum_{j=i}^{d+1}(-1)^j\binom{d+1}{j}\right](-1)^i\beta_{i-1}\\
&=\sum_{i=0}^{d+1}\binom{d}{i-1}\beta_{i-1}.
\end{align*}
Now recall the Poincar\'e polynomials of $G_k(\RR^n)$. Positivity of the coefficients of the $q$-analogs of binomial coefficients implies that the $(2n-8)$-th Betti number is nonzero for $n$ even and the $(2n-10)$-th Betti number is nonzero for $n$ odd. Therefore for $n$ even we have
\[
f_d\geq\binom{k(n-k)}{2n-8}
\]
and for $n$ odd we have
\[
f_d\geq\binom{k(n-k)}{2n-10}.
\]
In both cases, these are exponential in $n$.
\end{proof}

\begin{remark}
Using similar estimates, one can also show that the minimal number of faces $f_{d-j}$ in $G_k(\RR^n)$ for a fixed $k$ and a fixed codimension $j$ must grow exponentially as $n\to\infty$.
\end{remark}

\end{document}